\documentclass[11pt]{amsart}
\usepackage{hyperref}
\hypersetup{colorlinks,linkcolor=blue,citecolor=blue,filecolor=blue,urlcolor=blue}

\usepackage{anysize} \marginsize{1.3in}{1.3in}{1in}{1in}
\usepackage{amsfonts}

\usepackage[all,cmtip]{xy}
\usepackage{amsmath}
\usepackage{amsthm}
\usepackage{amssymb}
\usepackage{mathrsfs}
\usepackage{enumitem}
\newtheorem{thm}{Theorem}[section]

\newtheorem{prop}[thm]{Proposition}
\newtheorem{defn}[thm]{Definition}
\newtheorem*{definition*}         {Definition}
\newtheorem{lemma}[thm]{Lemma}

\newtheorem{cor}[thm]{Corollary}

\theoremstyle{remark}

\newcommand*{\Q}{\mathbb{Q}}
\newcommand*{\HH}{\mathbb{H}}

\newcommand*{\Z}{\mathbb{Z}}
\newcommand*{\G}{\mathbf{G}}
\newcommand*{\MT}{\mathbf{MT}}
\newcommand*{\bTheta}{\mathbf{\Theta}}
\newcommand*{\M}{\mathbf{M}}
\newcommand*{\R}{\mathbb{R}}

\newcommand*{\C}{\mathbb{C}}
\newcommand*{\g}{\frak{g}}
\newcommand*{\h}{\frak{h}}

\newcommand*{\FF}{\mathcal{F}}

\newcommand*{\mult}{\textrm{mult}}

\newcommand*{\ra}{\rightarrow}

\def\GL{{\rm GL}}

\def\BAut{{\bf Aut}}

\renewcommand{\d}{\partial}

\newcommand{\ddb}{\partial\bar{\partial}}

\renewcommand{\phi}{\varphi}

\def\codim{\operatorname{codim}}

\def\GL{\operatorname{GL}}

\def\tr{\operatorname{tr}}

\def\mult{\operatorname{mult}}
\def\vol{\operatorname{vol}}

\def\Imag{\operatorname{Im}}

\def\Real{\operatorname{Re}}

\def\ad{\operatorname{ad}}

\def\trans{\operatorname{trans}}
\def\Imag{\operatorname{Im}}

\def\pr{\operatorname{pr}}

\newcommand{\odd}{\mathrm{odd}}
\newcommand{\horiz}{\mathrm{horiz}}

\renewcommand{\bar}[1]{\overline{#1}}

\usepackage{amscd,amssymb,amsmath}
\usepackage{color}

\title[The Hodge-theoretic Ax--Schanuel conjecture]{The Ax--Schanuel conjecture for variations of Hodge structures} 

 \author[B. Bakker]{Benjamin Bakker}

\address{\noindent B. Bakker:  Dept. of Mathematics, University of Georgia, Athens, USA.}
\email{bakker@math.uga.edu}
\author[J. Tsimerman]{Jacob Tsimerman}

\address{\noindent J. Tsimerman:  Dept. of Mathematics, University of Toronto, Toronto, Canada.
}
\email{jacobt@math.toronto.edu}

\begin{document}
\begin{abstract}We extend the Ax--Schanuel theorem recently proven for Shimura varieties by Mok--Pila--Tsimerman to all varieties supporting a pure polarized integral variation of Hodge structures.  The essential new ingredient is a volume bound on Griffiths transverse subvarieties of period domains.\end{abstract}
\maketitle

\section{Introduction}

\subsection{History}

Motivated by arithmetic considerations, there has recently been much work in the setting of functional transcendence, and specifically on generalizations of the famous Ax--Schanuel theorem on
the exponential function to the context of hyperbolic uniformizations. Indeed, the strategy of Pila and Zannier for proving the Andr\'e--Oort conjecture is reliant on a functional transcendence result dubbed the `Ax-Lindemann theorem' by Pila.  The approach originates in the celebrated paper \cite{P}, where Pila used his counting theorem with Wilkie to establish the result in the case of the Shimura variety $X(1)^n$, for $ n\geq 1$. 

The Ax-Lindemann theorem 
was finally established in full generality for Shimura varieties in \cite{KUY} by Klingler, Ullmo, and Yafaev, and for mixed Shimura varieties by Gao \cite{Gao}. Motivated by an analogous (though much more difficult to carry out) approach to the more general Zilber--Pink conjectures, Mok, Pila, and the second author recently proved the full Ax--Schanuel conjecture for general Shimura varieties \cite{MPT}.
In this paper we prove the Ax--Schanuel conjecture in the more general setting of variations of (pure) Hodge structures (formulated recently by Klingler \cite[Conjecture 7.5]{klingler}). This is motivated largely by a recent
approach of Lawrence--Venkatesh to establishing arithmetic Shafarevich-like theorems for large classes of varieties, which seems to require the theorem we prove to work in full generality.

\subsection{Statement of Results}

Let $\mathbf{S}$ be the Deligne torus $\textrm{Res}_{\C/\R}\mathbb{G}_m$. Given a pure polarized Hodge structure $h:\mathbf{S}\to \BAut(H_\Z,Q_\Z)$, the Mumford-Tate group $\MT_h\subset\BAut(H_\Z,Q_\Z)$ is the $\Q$-Zariski closure of $h(\mathbf{S})$.  The associated Mumford--Tate domain $D(\MT_h)$ is the $\MT_h(\R)$-orbit of $h$ in the full period domain of polarized Hodge structures on $(H_\Z,Q_\Z)$.  By a \emph{weak Mumford--Tate domain} $D(\M)$ we mean the $\M(\R)$-orbit of $h$ for some normal $\Q$-algebraic subgroup $\M$ of $\MT_h$.

Let $X$ be a smooth algebraic variety over $\C$ of dimension $n$ supporting a pure polarized integral variation of Hodge structures $\mathscr{H}_\Z$.  Let $\MT_{\mathscr{H}_\Z}$ be the generic Mumford--Tate group, and let $\Gamma\subset \MT_{\mathscr{H}_\Z}(\Z)$ be the image of the monodromy representation $\pi_1(X)\to \MT_{\mathscr{H}_\Z}(\Z)$ after possibly passing to a finite cover. Let $\G$ be the identity component of the $\Q$-Zariski closure of $\Gamma$.  Let $D=D(\G)$ be the associated weak Mumford--Tate domain and $\phi: X\to \Gamma\backslash D$ the period map of $\mathscr{H}_\Z$.  The compact dual $\check D$ of $D$ is a projective variety containing $D$ as an open set in the archimedean topology.  

Consider the fiber product
\[
\xymatrix{
X\times D&W\ar@{}[l]|{\supset\hspace{-.7em}}\ar[r]^{\tilde \phi}\ar[d]&D\ar[d]^\pi\\
&X\ar[r]_\phi&\Gamma\backslash D.
}
\]  
In this situation, for any weak Mumford--Tate subdomain $D'=D(\M')\subset D$ such that $\Gamma\cap \M'(\Q)$ is $\Q$-Zariski dense, $\phi^{-1}\pi(D')$ is an algebraic subvariety of $X$ by a result of Cattani--Deligne--Kaplan \cite{alghodge}, and we refer to such subvarieties as \emph{weak Mumford--Tate subvarieties} of $X$.  
 
\begin{thm}[Ax--Schanuel for variations of Hodge structures]\label{main}In the above setup, let $V\subset X\times \check D$ be an algebraic subvariety, and let $U$ be an irreducible analytic component of $V\cap W$ such that
\[\codim_{X\times \check D}(U)<\codim_{X\times\check D}(V)+\codim_{X\times\check D}(W).\]
Then the projection of $U$ to $X$ is contained in a proper weak Mumford--Tate subvariety.
\end{thm}
The theorem for example implies that the (analytic) locus in $X$ where the periods satisfy a given set of algebraic relations must be of the expected codimension unless there is a reduction in the generic Mumford--Tate group.   See \cite{klingler} for some related discussions.

\subsection{Outline of the proof}

We follow closely the proof in \cite{MPT}. There are two serious complications that have to be addressed, which are as follows:

First, we need to find a suitable fundamental domain in $D$ for the image of $X$ in $\Gamma\backslash D$. This domain has to be definable in the o-minimal structure $\R_{an,exp}$, and have 
certain growth properties. In the Shimura case, this is done by using a Siegel set. In our current setup this seems more difficult, due to the absence of toroidal co-ordinates. Instead, we use
Schmid's theory of degenerations of Hodge structures to define our fundamental domain, which also provides a new approach in the setting of Shimura varieties. For more details on this,
see \S\ref{definable}.

Second, the proof of Theorem \ref{main} requires a volume bound on Griffiths transverse\footnote{It is essential to restrict to Griffiths transverse subvarieties, as the general statement is false since, for example, $D$ contains compact subvarieties.} subvarieties $X\subset D$ analogous to those proven by Hwang--To for hermitian symmetric domains \cite{hwangto1}. We prove this in \S\ref{volume} and the result is as follows:

\begin{thm}\label{volumebound}There are constants $\beta,\rho>0$ (only depending on $D$) such that for any $R>\rho$, any $x\in D$, and any positive-dimensional Griffiths transverse global analytic subvariety $Z\subset B_x(R)\subset D$,
\[\vol(Z)\gg e^{\beta R}\mult_xZ\]
where $B_x(R)$ is the radius $R$ ball centered at $x$ and $\vol(Z)$ the volume with respect to the natural left-invariant metric on $D$.
\end{thm}

In \S\ref{heights} we establish all the required comparisons between the various height and distance functions that show up, and \S\ref{proof} completes the proof.

\subsection*{Acknowledgements}  The first author was partially supported by NSF grant DMS-1702149.

\section{Volume estimates}\label{volume}
In this section we prove Theorem \ref{volumebound}; we begin with some general remarks.  Without loss of generality, we may clearly assume $D$ is a full period domain.  Further, letting $\HH$ be the upper half-plane, $D\times \HH$ embeds isometrically into a period domain $D'$ of weight one larger by tensoring with the weight one Hodge structure of an elliptic curve, and it therefore suffices to consider $D$ of odd weight.  We make both of these assumptions for the remainder of this section.  For general background on period domains and Hodge structures, see for example \cite{perioddomain}.
\subsection{Hodge norms}
A point $x\in D$ yields a Hodge structure $H_x$ on $H_\Z$ polarized by $Q_Z$.  Recall that the Hodge metric $h_x(v,w)=Q_\Z(v,C_x\bar w)$ is postive-definite, where $C_x$ is the Weil operator of $H_x$.  For any $w\in H_\C$ we can define a function $h_x(w):=h_x(w,w)$ on $D$.  Note that  $g^*h_x(w)=h_x(g^{-1}w)$ for $g\in\G(\R)$.  Recall also that a choice of point $x\in D$ naturally endows the Lie algebra $\g_\R$ of $\G(\R)$ with a weight zero Hodge structure $\g_x$ polarized by the Killing form, and that the holomorphic tangent space at $x$ is naturally identified with $\g_x^-$, where as usual we give $\g^{p,-p}$ grading $p$.  We refer to the odd part of $\g_x^-$ as the horizontal directions, and to $\g_x^{-1,1}$ as the Griffiths transverse directions.

\begin{lemma}\label{derivatives}  For Hodge-pure horizontal (in particular Griffiths transverse) directions $X\in \g_x^-$, we have
\begin{align*}
\d h_x(w)(X) &= -2 h_x(Xw,w)\\
\ddb h_x(w)(X,\bar X)&=  2h_x(Xw)+2h_x(\bar X w)
\end{align*}
\end{lemma}
\begin{proof}
Note that in $\C[z,\bar{z}]/(z^2,\bar{z}^2)$ we have
\begin{align*}
&\exp(-zX)\exp\left(zX+\bar{z}\bar{X}+\frac{|z|^2}{2}\left([X,\bar X]^{<0}+[\bar X,X]^{>0}\right)\right)=\\
&=\left(1-zX\right)\left(1+zX+\bar{z}\bar{X}+\frac{|z|^2}{2}\left([X,\bar X]^{<0}+[\bar X,X]^{>0}\right)+\frac{|z|^2}{2}\left(X\bar{X}+\bar{X}X\right)\right)\\
&=1+\bar{z}\bar{X}+|z|^2\left(-X\bar{X}+\frac{1}{2}\left([X,\bar X]^{<0}+[\bar X,X]^{>0}\right)+\frac{1}{2}\left(X\bar{X}+\bar{X}X\right)\right)\\
&=1+\bar{z}\bar{X}+\frac{1}{2}\left(-[X,\bar{X}] +[X,\bar X]^{<0}+[\bar X,X]^{>0}\right)\\
&=1+\bar{z}\bar{X}+\frac{1}{2}\left(-[X,\bar X]^{\geq 0}+[\bar X,X]^{>0}\right)
\end{align*}
which is in the parabolic stabilizing the Hodge flag at $x$.  Thus, modulo $(z^2,\bar z^2)$ we have
\[\exp(zX).x= \exp\left(M(z X,\bar z\bar X)\right).x\]
where $M(zX,\bar z\bar X)=zX+\bar{z}\bar{X}+\frac{|z|^2}{2}\left([X,\bar X]^{<0}+[\bar X,X]^{>0}\right)\in \g$.  Thus,
\begin{align*}
\d h_x(w)(X) &= \frac{\d}{\d z} \exp(zX)^*h_x(w)|_{z=0}\\
&= \frac{\d}{\d z} h_0\left(\exp\left(-M(z X,\bar z\bar X)\right).w\right)|_{z=0}\\
&=h_x(-Xw,w)+h_x(w,-\bar X w)\\
&=-2h_x(Xw,w)
\end{align*}
where we have used that $X$ is horizontal and thus conjugate self-adjoint with respect to $h_x$.  Likewise,
\begin{align*}
\ddb h_x(w)(X,\bar X) &= \frac{\d^2}{\d z\d \bar z} \exp(zX)^*h_x(w)|_{z=0}\\
&= \frac{\d^2}{\d z\d \bar z} h_0\left(\exp\left(-M(z X,\bar z\bar X)\right).w\right)|_{z=0}\\
&=h_x(-Xw,-Xw)+h_x(-\bar X w,-\bar X w)\\
&+\Real h_x(-[X,\bar X]^{<0}w,w)+\Real h_x(-[\bar X,X]^{>0}w,w)\\
&+\Real h_x((X\bar X+\bar X X)w,w)\\
&=2h_x(Xw)+2h_x(\bar Xw)
\end{align*}
where we have used that $[X,\bar X]^{<0}=[\bar X,X]^{>0}=0$ since $X$ is Hodge-pure, as well as the conjugate self-adjointness of $X$.
\end{proof}
\subsection{Distance functions}
Let $\pi: D\to D_W$ be the projection to the associated symmetric space by taking the Weil Jacobian.  For every $x\in D$, we denote the Weil Hodge structure on $\g_\C$ by $\h_x$.  Note that both Hodge structures $\g_x$ and $\h_x$ induce the same Hodge metric on $\g_\C$.  Given $x_0\in D$, $\pi$ is identified with $\G(\R)/V\to \G(\R)/K$, where $V$ is the stabilizer of $x_0$ under $\G(\R)$ and $K$ is the unitary subgroup of $\G(\R)$ with respect to $h_{x_0}$.  $K$ is a maximal compact subgroup of $\G(\R)$.

Let $v_0$ be a unit-length generator of $\det \h_{x_0}^{+}$ in $\bigwedge^{\dim D_W}\h_{x_0}$, and define a function $\phi_0:D\to \R$ by
\[\phi_0(x):=\log h_{x}(v_0) \]
$\phi_0$ factors through the projection $\pi$.  If $F_0$ is the fiber of $\pi$ containing $x_0$, then by the $KAK$ decomposition of $\G(\R)$, $\phi_0$ in fact only depends on $F_0$ since $K$ fixes $v_0$ up to a phase.

\begin{lemma}\label{nonzero} $i\ddb \phi_0$ is strictly positive on Griffiths transverse tangent directions at $x_0$.
\end{lemma}

\begin{proof}  Let $X\in \g_{x_0}^{-1,1}$, and note that $X\in \h_{x_0}^{-1,1}\oplus \h_{x_0}^{1,-1}$.  Let $X^{-1,1},X^{1,-1}$ be the graded pieces of $X$ with respect to the Weil Hodge structure.  Fixing a basis $Y_i$ of $\h_{x_0}^+$, we see that
\[\ad(X)\left(Y_1\wedge\cdots\wedge Y_k\right)=\sum_i (-1)^{i-1}Y_1\wedge\cdots\wedge \ad(X^{-1,1})Y_i\wedge\cdots\wedge Y_k. \]
The vectors on the right-hand side are all linearly independent, so if $\ad(X)v_0=0$ then $\ad(X)\h_{x_0}^+=0$.  Likewise, if $\ad(\bar X)v_0=0$, then $\ad(X)\h_{x_0}^-=0$.  Thus, if $i\ddb\phi_0(X,\bar X)=0$ then by Lemma \ref{derivatives} $\ad(X)$ kills $\h_{x_0}^{\odd}$ and in particular $\bar X$, but this implies $X=0$ \cite[Corollary 12.6.3(iii)]{perioddomain}.
\end{proof}

Define the horizontal distance from $x$ to $x_0$, denoted $d_{0}^\horiz(x)$, to be the geodesic distance between $y:=\pi(x)$ and $y_0:=\pi(x_0)$ with respect to the natural $\G(\R)$-invariant metric on the symmetric space $D_W$.  Let $A$ be an $\R$-split torus of $\G(\R)$ that is Killing-orthogonal to $K$.  By the $KAK$ decomposition of $\G(\R)$, the distance $d^{D_W}_{0}(y)$ and $\phi_0(x)$ are both determined by $d^{D_W}_{0}(a y_0)$ and $\phi_0(ax_0)$ for $a\in A$.    $A y_0$ is evidently a totally geodesic submanifold of $D_W$, and the restriction of the invariant metric is a Euclidean metric in exponential coordinates, so 
\begin{equation}\label{dist}
d^{D_W}_{0}(ay_0)^2\sim\sum_{i}t^2_i
\end{equation}
where $a=\exp(\sum_i t_iT_i)$ for some chosen basis $T_i$ of the Lie algebra $\frak{a}$ of $A$.

The main result of this subsection is the following comparison.  Note that both $d_0^{\horiz}$ and $\phi_0$ vanish exactly on $F_0$.

\begin{prop}\label{comparison} $d^\horiz_0(x)\ll \phi_0(x)+O(1)$ and $\phi_0(x)\ll d^\horiz_0(x)+O(1)$.
\end{prop}

\begin{proof}  Griffiths--Schmid \cite{GS} show that a function closely related to $\phi_0$ is an exhaustion function of $D$.  For $D_W$, their function is
\[\phi_0'(gy_0) = h_{x_0}(gv_0)\]
and their result implies $\phi_0\to \infty$ at the boundary of $D$.

Now, consider the decomposition
$$v_0=\sum_\alpha v_\alpha $$
by $\frak{a}$-weights.  Note that as $A$ is Killing-orthogonal to $K$, $\frak{a}$ is odd and therefore self adjoint with respect to $h_{x_0}$.  It follows then that the decomposition of $\bigwedge^{\dim D_W}g_{\C}$ into $\frak{a}$-weight spaces is orthogonal with respect to $h_{x_0}$, and thus for $T\in \frak{a}$,
\[h_{x_0}(\exp(-T)v_0)=\sum_{\alpha}e^{-2\alpha(T)}h_{x_0}(v_\alpha).\]
Let $\Xi\subset\frak{a}$ be the convex hull of the $\alpha$ for which $v_\alpha\neq 0$.  Since $\phi_0\to\infty$ at the boundary, we must have $0\in\Xi$, for otherwise there would be a direction in which $\phi_0$ is bounded.  It then follows that

\[\log\sum_i \left(e^{T_i^\vee}(a)+e^{-T_i^\vee}(a)\right)\ll\phi_0\left(ax_0\right)\ll \log\sum_i \left(e^{T_i^\vee}(a)+e^{-T_i^\vee}(a)\right)\]
which implies the claim by \eqref{dist}.
\end{proof}

\subsection{Multiplicity bounds}
For any $r>0$ and $x_0\in D$, denote by 
\[B^{\phi_{0}}(r):=\{x\in D\mid \phi_0(x)<r\}\]
and for any Griffiths transverse analytic subvariety $Z\subset D$,
\[\vol^{\phi_0}(Z):=\int_{Z}i\partial\bar\partial\phi_0.\]
\begin{prop}\label{formbounds}  Let $\omega$ be the K\"ahler form of the natural left-invariant hermitian metric on $D$.
\begin{enumerate}
\item $i\ddb\phi_0\geq_{\trans} 0$ and $i\ddb \phi_0=O_{\trans}(\omega)$;
\item $|\d \phi_0|^2 =O_{\trans}(i\ddb \phi_0)$.
\end{enumerate}
\end{prop}
In the statement of the proposition, the notations $O_{\trans}(\cdot)$ and $\geq_{\trans}$ mean the bound holds in Griffiths transverse tangent directions.
\begin{proof}
By definition, $\omega_x(X,\bar X) \sim h_x(X)$.  For horizontal $X$, $\tr(X\bar X)\sim h_x(X)$ is larger than the maximum eigenvalue of $X^*h_x$ with respect to $h_x$.  For $X\in \g^-$ pure, by Lemma \ref{derivatives} we have
 \[i\ddb\phi_{0}(X,\bar X)=2\left (\frac{h_x(Xv_0)}{h_x(v_0)}+\frac{h_x(\bar X v_0)}{h_x(v_0)}\right)-\left|\frac{h_x(Xv_0,v_0)+h_x(v_0,\bar X v_0)}{h_x(v_0)}\right|^2\] 
which is nonnegative by the triangle inequality and bounded by the maximal eigenvalue of $X^*h_x$ with respect to $h_x$, so (1) follows.  

The second claim follows by Lemma \ref{derivatives} and the following lemma:
\begin{lemma}\label{CS}There is a $\beta>0$ (only depending on $D$) such that for any $x\in D$, $w\in H_\C$, and $X\in\g_x^{-1,1}$,
\[h_x(w)\cdot \frac{h_x(Xw)+h_x(\bar X w)}{2}\geq (1+\beta)\left|h_x(Xw,w)\right|^2.\]
\end{lemma}

\begin{proof}

Let $w=\sum_i w^{i,n-i}$ be the decomposition into Hodge components at $x$, so that we have Hodge decompositions $Xw=\sum_i Xw^{i,n-i}$, $\bar Xw=\sum_i \bar Xw^{i,n-i}$.  

Now let
$$ a_i^2=h_x(w^{i,n-i}),\quad b_{i-1}^2=h_x(Xw^{i,n-i}), \quad c_{i+1}^2=h_x(\bar X w^{i,n-i}),$$ and we'll also set $b_n=c_0=0$.  Note that since $X$ and $\bar X$ are adjoint we have 

$$ h_x(Xw,w)=\sum_i h_x(Xw^{i+1,n-i-1},w^{i,n-i})$$ and $$|h_x(Xw^{i+1,n-i-1},w^{i,n-i})|=|h_x(w^{i+1,n-i-1},\bar Xw^{i,n-i})|\leq \min(a_ib_i,a_{i+1}c_{i+1}).$$
Thus it is sufficient to show that $$ \left(\sum_{i=0}^n a_i^2\right)\left(\sum_{i=0}^{n-1} b_i^2 + \sum_{i=1}^n c_i^2 \right)\geq (2 + \delta)\left(\sum_{i=0}^n a_i(r_ib_i+s_ic_i) \right)^2$$ for some choice 
of nonnegative $r_i,s_i$ with $r_i+s_{i+1}=1$ for $0\leq i\leq n-1$. By the Cauchy--Schwartz inequality, the left-hand side is greater than or equal to $\left(\sum_{i=0}^na_i\sqrt{b_i^2+c_i^2}\right)^2$.  Thus, it 
suffices to show for each $i$,
\[b_i^2+c_i^2\geq (2+\delta)\left(r_ib_i+s_ic_i\right)^2.\]

Note that $x^2+y^2 - 2(rx+sy)^2$ is positive definite if $(1-2r^2)(1-2s^2)>4r^2s^2$. 

\begin{lemma}

There exist non-negative real numbers $r_0,s_1,r_1,s_2,\dots,s_{n-1},r_{n-1},s_n$, with $r_i+s_{i+1}=1$ for $0\leq i\leq n-1$, $\max(r_0,s_n)<\frac{1}{\sqrt{2}}$, and
$(1-2r_i^2)(1-2s_i^2)>  4r_i^2s_i^2$ for all $1\leq i \leq n-1$.

\end{lemma}

\begin{proof}

Note that at $r_i=s_i=\frac12$ we get exact equality, in that $(1-2r_i^2)(1-2s_i^2)= 4r_i^2s_i^2$. Thus, we set $r_j=\frac12+\delta_j$, where $\delta_0=\frac19$ and $\delta_{j+1}$ is sufficiently
small in terms of $\delta_j$ to ensure $(1-2r_{j+1}^2)(1-2s_{j+1}^2) >  4r_j^2s_j^2$.  

\end{proof}

The statement now follows by picking the $r_i,s_i$ from the previous lemma, and setting $(2+\delta)$ to be the largest number such that $x^2+y^2 - (2+\delta)(r_ix+s_iy)^2$ is positive semi-definite for
$1\leq i\leq n-1$ and $1-(2+\delta)s_0^2$ is nonnegative.

%Now suppose that we can satisfy
% $b_i^2+c_i^2 \geq 2(r_ib_i+s_ic_i)^2$ 
% 
% we need  $$(1-2r_i^2)*(1-2s_i^2)\geq  r_is_i.

\end{proof}

\end{proof}

The previous proposition implies that the volume of Griffiths transverse subvarieties of $D$ grows at least exponentially in the radius:

\begin{prop} There is a constant $\beta>0$ such that for any $R>0$ and any positive-dimensional Griffiths transverse global analytic subvariety $Z\subset B^{\phi_0}(R)$, 
\[e^{-\beta r}\vol^{\phi_0}(Z\cap B^{\phi_0}(r))\]
is a nondecreasing function in $r\in[0,R]$.
\end{prop}
\begin{proof}  Let $\psi_0 =-e^{-\beta\phi_0}$ for $\beta$ the constant from Lemma \ref{CS}.  We have
\[i\ddb \psi_0=\beta e^{-\beta\phi_0}\left(i\ddb \phi_0-\beta|\d\phi_0|^2\right)\]
which is nonnegative in Griffiths transverse directions by the proof of Proposition \ref{formbounds}(ii).  By Stokes' theorem we have
 \begin{align*}
 \vol^\phi (Z\cap B^{\phi_{0}}(r))&= \int_{Z\cap B^{\phi_{0}}(r)}(i\ddb \phi_0)^d\\
 &=\int_{Z\cap\partial B^{\phi_{0}}(r)}d^c \phi_0\wedge (i\ddb \phi_0)^{d-1} \\
 &=\beta^{-1}e^{\beta r}\int_{Z\cap\partial B^{\phi_{0}}(r)}d^c \psi_0\wedge (i\ddb \phi_0)^{d-1}\\ 
 &=\beta^{-1}e^{\beta r}\int_{Z\cap B^{\phi_{0}}(r)}i\partial\bar\partial \psi_0\wedge (i\ddb \phi_0)^{d-1}\\
 &=\beta^{-d}e^{\beta dr}  \int_{Z\cap B^{\phi_{0}}(r)}(i\ddb \psi_0)^d
 \end{align*}
 which implies the claim, as $\psi_0|_Z$ is plurisubharmonic.
\end{proof}

\begin{proof}[Proof of Theorem \ref{volumebound}]  Choose a fixed euclidean ball $B$ centered around $x_0$.  By a classical result Federer (see for example \cite{Stolzenberg}), we have an inequality of the form $\vol^{\mathrm{eucl}}(Z\cap B)\gg \mult_{x_0}Z$.  Choose a fixed radius $r_0$ such that $B\subset B^{\phi_0}(r_0)$.  After possibly shrinking $B$, $i\ddb\phi_0$ is comparable to the euclidean K\"ahler form on $B$ in Griffiths transverse directions by Lemma \ref{nonzero}, and combining this with the above proposition we have 
$$\vol^{\phi_0}(Z\cap B^{\phi_0}(r))\gg e^{\beta r}\vol^{\phi_0}(Z\cap B^{\phi_0}(r_0))\gg e^{\beta r}\mult_{x_0}Z$$
for $r>r_0$.  By Proposition \ref{comparison}, the balls $B^{\phi_{0}}(r)$ are comparable to the balls $B^\horiz_{x_0}(r)$, which are in turn comparable to the balls $B_{x_0}(r)$ with respect to a left-invariant metric on $D$ for $r\gg0$, so we obtain the bound in Theorem \ref{volumebound}.

\end{proof}
\section{Definable fundamental sets}\label{definable}

Throughout the following, by definable we mean definable with respect to the o-minimal structure $\R_{\mathrm{an,exp}}$.  Let $X$ be a smooth algebraic variety supporting a pure polarized integral variation of Hodge structures $\mathscr{H}_\Z$, and let $(\bar X,E)$ be a proper log-smooth compactification of $X$.  For simplicity we may assume that $\mathscr{H}_\Z$ has unipotent local monodromy and that the associated period map $\phi:X\to \Gamma\backslash D$ is proper, although the argument carries through without making these assumptions.  We may also assume that the monodromy $\Gamma$ is torsion-free.

The structure of $X$ as an algebraic variety canonically endows it with the structure of a definable manifold, and the choice of  compactification $(\bar X,E)$ allows us to choose a definable atlas of $X$ of finitely many polydisks $\Delta^k\times(\Delta^*)^\ell$.  Note that any polydisk chart $P$ in such an atlas $\{P_i\}$ can be shrunk to yield a new such atlas, as the complement of $\bigcup_{P_i\neq P}P_i$ is contained in $P$ and has compact closure in the interior closure of $P$ in $\bar X$.  Let
\[\exp:\Delta^k\times\HH^\ell\to \Delta^k\times(\Delta^*)^\ell\] 
be the standard universal cover, and choose a bounded vertical strip $\Sigma\subset \HH$ such that $\Delta^k\times\Sigma^\ell$ is a fundamental set for the action of covering transformations.  By the above remark, by shrinking a polydisk we may always restrict to a region in $\Delta^k\times\Sigma^\ell$ where $|z_i|$ is bounded away from $1$ on the $\Delta$ factors and $\Imag z_i$ is bounded away from $0$ on the $\Sigma$ factors.  

Choose lifts $\tilde\phi:\Delta^k\times\HH^\ell\to D $ of the period map restricted to each chart, and let $\mathcal{F}$ be the disjoint union of $\Delta^k\times \Sigma^\ell$ over all charts.  We then have a diagram
\begin{equation}
\xymatrix{
\mathcal{F}\ar[r]^{\tilde\phi}\ar[d]_\exp& D \\
X&}\label{diagram}\end{equation}
and $\mathcal{F}$ has a natural definable structure.

Note that the embedding $ D \subset\check{ D }$ gives $ D $ a canonical definable structure.
\begin{lemma}\label{definablelift} Both maps in \eqref{diagram} are definable.  
\end{lemma}

\begin{proof}  The claim for the vertical map is obvious.  By the nilpotent orbit theorem, for each polydisk $\tilde\phi=e^{zN}\tilde\psi$ where $\tilde\psi=\psi\circ\exp$ for some extendable holomorphic function $\psi:\Delta^{n}\to D $ (after shrinking the polydisks).  The action of $\G(\R)$ on $ D $ is definable, and $e^{z\cdot N}$ is polynomial in $z$, so $\tilde\phi:\Delta^k\times\Sigma^\ell\to D $ is definable. 
\end{proof}

Fix a left-invariant metric $h_ D $ on $ D $ and let $\Phi = \tilde\phi(\mathcal{F})$.

\begin{prop}\label{smallvolume}  Let $Z\subset \check D $ be algebraic.  For all $\gamma\in \G(\Z)$, $\vol(Z\cap\gamma\Phi)= O(1)$.
\end{prop}
\begin{proof}Evidently it is enough to show $\vol(Z'\cap\Phi)= O(1)$ for all $Z'$ in the same connected component of the Hilbert scheme of $\check D $ as $Z$.  Further, it suffices to show $\vol(\tilde\phi^{-1}(Z')\cap\Delta^k\times\Sigma^\ell)= O(1)$ for each lifted polydisk chart $\tilde\phi:\Delta^k\times\HH^\ell\to D $, where the volume is computed with respect to $\tilde\phi^*h_ D $.

For any holomorphic horizontal map $f:M\to \Gamma\backslash D $ we have $f^*h_ D \ll \kappa_M$ where $\kappa_M$ is the Kobayashi metric of $M$.  In particular, for $M=\Delta^k\times\HH^\ell$ the metric $\kappa_M$ is the maximum over the coordinate-wise Poincar\'e metrics.  After shrinking the polydisk, the factors in $\Delta^k\times \Sigma^\ell$ have finite volume with respect to the Kobayashi metric of the larger polydisk, and thus it is enough to uniformly bound the degree of the projection of $\tilde\phi^{-1}(Z')$ to any subset of coordinates. 

By definable cell decomposition,  for any definable subset $L\subset\R^N$ and any coordinate projection $\R^N\to\R^M$, the number of connected components in the fibers of $L$ is bounded.  Applying this to the universal family of $\tilde\phi^{-1}(Z')\subset \Delta^k \times\Sigma^\ell$, the claim follows.

\end{proof}
\section{Heights}\label{heights}

Fix a basepoint $x_0\in \Phi$ so that we have an identification $ D \cong \G(\R)/V$ for a compact subgroup $V\subset \G(\R)$.  Thinking of $ D $ as a space of Hodge structures on the fixed integral lattice $(H_\Z,Q_\Z)$, as before we denote by $h_x$ the induced Hodge metric on $H_\C$ corresponding to $x\in D$.   

\begin{defn} For $\gamma\in \G(\Z)$ let $H(\gamma)$ be the height of $\gamma$ with respect to the representation $\rho_\Z:\G(\Z)\to \GL(H_\Z)$. For $g\in \G(\R)$, we denote by 
$||\rho_\R(g)||$ the maximum archimedean size of the entries of $\rho_\R(g)$, so that if $\gamma\in\G(\Z)$ we have $H(\gamma)=||\rho_\R(\gamma)||$.
\end{defn}

For any $R>0$ let $B_{x_0}(R)\subset D $ be the ball of radius $R$ centered at $x_0$.  The main goal of this section is to establish the following:
\begin{thm}\label{smallheight} For any $R>0$, every element $\gamma$ of
\[\{\gamma\in \G(\Z)\mid B_0(R)\cap \gamma^{-1}\Phi\neq \varnothing\}\]
has height $H(\gamma)=e^{O(R)}$.
\end{thm}
Define $d_0(x)=d(x,x_0)$.  We write $f\preceq g$ if $|f|\ll |g|^{O(1)}+O(1)$, and $f\asymp g$ if $f\preceq g$ and $g\preceq f$.
\begin{lemma}\label{distanceheight}Let $\lambda(x,x')$ be the maximal eigenvalue of $h_x$ with respect to $h_{x'}$.  Then
\begin{enumerate}
\item For all $g\in \G(\R)$ we have $||\rho_\R(g)||\asymp e^{d_0(gx_0)}$;
\item $\lambda(x,x')\asymp e^{d(x,x')}$.
\end{enumerate}
\end{lemma}
\begin{proof}Choose a maximal compact subgroup $K\subset \G(\R)$ containing $V$ and a left-invariant metric on the associated symmetric space $\G(\R)/K$.  Note that the diameters of the fibers of $\G(\R)/V\to \G(\R)/K$ are bounded.  Choosing a split maximal torus $A\subset \G(\R)$ and a basis $A_i$ of the Lie algebra $\mathfrak{a}$ of $A$, we have for any $g\in \G(\R)$ with $KAK$ decomposition $g=k_1a k_2$
\[\sqrt{\sum_i t_i^2}\ll d_0(gx_0)=d_0(ax_0)+O(1)\ll\sqrt{\sum_i t_i^2}+O(1)\]
where $a=\exp(\sum_it_iA_i)$. As 
$$\max_i \exp(|t_i|)\preceq\rho_\R(g)\preceq \max_i \exp(|t_i|)$$
part (1) follows.

For part (2), note that by $\G(\R)$-invariance we may restrict to the case $x'=x_0$.   Setting $\rho=\rho_\R$ for convenience, note that $\tr(\rho(g)^*\rho(g))$ is a sum of the eigenvalues of $h_{gx_0}$ wrt $h_{x_0}$, where $\rho(g)^*$ is the adjoint of $\rho(g)$ wrt $h_{x_0}$.  Thus $\tr(\rho(g)^*\rho(g))\asymp \lambda(gx_0,x_0)$. As $\tr(\rho(g)^*\rho(g))$ is the sum of the squares of the entries of $\rho(g)$, part (2) follows from part (1).

% Provided we can choose $A$ so that the weight vectors of $\rho_\C$ are orthogonal with respect to $h_0$, $\tr(\rho(g)^*\rho(g))$ is a sum of monomials in the $e^{t_i}$, where $\rho(g)^*$ is the adjoint of $\rho(g)$ with respect to $h_0$.  Subject to this assumption, then, the two claims follow.
%
%To prove that such an $A$ exists, first note that over $\R$ the Hodge structure at $x_0$ splits as an orthogonal sum of simple Hodge structures with Hodge numbers $(1,0,\ldots,0,1)$.  In the odd weight case, it is easy to verify that such a torus exists for each factor, and the product torus will then suffice.  In the even weight case,
  \end{proof}

We define a proximity function of the boundary by the minimal period length:
\[\mu(x)=\min_{v\in H_\Z\backslash\{0\}} h_x(v).\]
For any $v\in H_\C$ we have $\log\frac{h_{x_0}(v)}{h_x(v)}\ll d_0(x)+O(1) $ by part (2) of Lemma \ref{distanceheight}, and so we deduce the following:
\begin{cor}\label{bound1} $-\log\mu \ll d_0+O(1)$.
\end{cor}
\begin{proof}  There is some $v\in H_\Z$ with $\log\mu=\log h_x(v)$ and thus
\[-\log\mu=-\log h_x(v)\ll\log \frac{h_{x_0}(v)}{h_x(v)}+O(1)\ll d_0(x) +O(1).\]
\end{proof}
When restricted to the fundamental set $\Phi$, we in fact have a comparison in the other direction:
\begin{lemma}\label{bound2}  For $x\in\Phi$ we have $d_0(x)\ll -\log\mu(x)+O(1)$.
\end{lemma}
\begin{proof}We may assume $\mathcal{F}$ is a single $\Delta^k\times\Sigma^\ell$.  After choosing logarithms  $N_1,\dots,N_\ell$ of the local monodromy operators of the variation over $\Delta^k\times (\Delta^*)^\ell$, let $v_i$ be a fixed basis of $H_\Z$ descending to a basis of the multi-graded module associated to the $\ell$ weight filtrations, where we take each grading centered at $0$.  Let $w_i^{(j)}$ for $j=1,\ldots,\ell$ be the weights of $v_i$ w.r.t. $N_j$.  By \cite{hodgeasymp}, for every permutation $\pi$ and on each region $S_{\pi}\subset\Delta^k\times \HH^\ell$ of the form $\Imag z_{\pi(1)}\gg \cdots \Imag z_{\pi(\ell)}\gg 1$ we have 
\[h_{\tilde\phi(z)}(v_i)\sim \left(\frac{\Imag z_{\pi(1)}}{\Imag z_{2}}\right)^{w_i^{(1)}}\cdots\hspace{1em}\left(\frac{\Imag z_{\pi(\ell-1)}}{\Imag z_{\pi(\ell)}}\right)^{w_i^{(\ell-1)}}\cdot\hspace{1em}(\Imag z_{\pi(\ell)})^{w_i^{(\ell)}}.\]
where ``$\sim$'' means ``within a bounded function of."  As the set of weights is preserved under negation, it follows that $\max_i h_{\tilde\phi(z)}(v_i)\sim (\min_i h_{\tilde{\phi}(z)}(v_i))^{-1}$, and so by Lemma \ref{distanceheight},
\[d_0(\tilde{\phi}(z))\ll \max_i\log h_{\tilde{\phi}(z)}(v_i)\ll -\log\mu(\tilde{\phi}(z))+O(1)\] uniformly on every such region. The $S_{\pi}$ can be made to cover the region $\Delta^k\times\Sigma^\ell$ after shrinking $\Sigma$, and the result follows.
\end{proof}
\begin{proof}[Proof of Theorem \ref{smallheight}] Suppose $x\in B_0(R)\cap \gamma^{-1}\Phi$ for $\gamma\in \G(\Z)$.  Putting together Lemma \ref{bound2} and Corollary \ref{bound1} we have
\[d_0(\gamma x)\ll -\log\mu(\gamma x)+O(1)=-\log\mu(x)+O(1)\ll d_0(x)+O(1)\]
 and since
 \[d_0(\gamma x_0)\leq d(\gamma x,\gamma x_0 )+d(\gamma x,x_0)\leq d_0(x)+d_0(\gamma x)\]
 we are finished by part (1) of Lemma \ref{distanceheight}.
\end{proof}

\section{ The proof of  Theorem \ref{main}}\label{proof}
The remainder of the proof follows the same general strategy as \cite{MPT}.  There are sufficiently many differences, however, that we include the necessary modifications.
 
Recall that $ D $ sits naturally as an open subset in its compact dual $\check{ D }$ which has the structure of a projective variety. Let $M$ be the Hilbert 
scheme of all subvarieties of $X\times\check{ D }$ with the same Hilbert polynomial as $V$.
 Moreover let $\mathcal{V}\ra M$ be the universal family 
over $M$, with a natural embedding $\mathcal{V}\hookrightarrow(X\times\check{ D })\times M$.  

Let $\mathcal{V}_W$ be the base-change to $W\times M$.  The action of $\Gamma$ on $X\times D $ lifts to $\mathcal{V}_W$, and we define $\mathcal{V}_X:=\Gamma\backslash \mathcal{V}_{ W}$, which is naturally an analytic variety.  Note that as $M$ is proper, $\mathcal{V}_W$ is proper over $W$, and likewise $\mathcal{V}_X$ is proper over $X$. 

We endow $\mathcal{V}_X$ with a definable structure as follows.  $\mathcal{V}$ is algebraic and has an induced definable structure.  By Lemma \ref{definablelift}, pulling back to $\FF\times M$ and quotienting out by the definable equivalence relation $\FF\ra X$ we obtain the desired definable structure on $\mathcal{V}_X$.  

%\begin{itemize}
%
%\item Let $L$ be the integral local system on $X$, and  $F^\cdot\subset H\otimes\Oo_X$ be the sheaves defining the Hodge filtration. The $F^{\cdot}$ are actually definable as subsheaves of 
%$H\otimes\Oo_X$  - which is to say, they have definable frames - as follows from Lemma \ref{definablelift}. 
%
%\item 
%
%\end{itemize}
%
%\begin{lemma}\label{definablehodge}
%
%$X^o_1$ is a definable subset of $G_{\dim V}(J_r(X\times X))$.
%\end{lemma}
%
%\begin{proof}
%
%By lemma \ref{definablelift} we have that the map from $\FF$ to $X$ is a definable surjection. It follows that the map from
% $G_{\dim W} (J_r(\Phi\times X))$ to $G_{\dim W}(J_r(\Gamma\backslash D \times X))$ is definable. Now $X^0_1$ is the image of the pull-back of $ D ^o_1$ to $\FF\times X$, which is a definable set. The claim follows. 
% 
% \end{proof}

%Thus, by the definable Chow theorem (Peterzil-Starchenko \cite{PS}), $X_1$ is naturally an algebraic variety.
% 
%We let $\FF$ be a definable (semialgebraic) fundamental domain for $\pi$, and $\FF_1$ be the pullback of $\FF$ to $ D _1$, so that $\FF_1$ is a fundamental domain for $\pi_1$.
%Since definable maps induce definable maps on Grassmanian  bundles and jet spaces, we see that the natural projection map $\pi_1: D _1\ra X_1$ is definable when restricted to 
%$\FF_1$.

Suppose the theorem is false for the sake of contradiction. Moreover, suppose that $\dim X$ is minimal, and subject to that assumption, $\codim V+\codim W-\codim U$ is as large as possible, and subject to that assumption, that $\dim U$ is maximal.

Define a closed analytic subvariety $T\subset \mathcal{V}_W$ consisting of all pairs $(p,V')$ such that $V'\cap W$ has dimension at least $\dim U$ around $p$, and let $T_0$ be the irreducible component containing $(p,V)$ for some (hence any) point $p\in U$.  Let $Y:=\Gamma\backslash T_0\subset \mathcal{V}_X$, which is a closed definable analytic subvariety.  Now, the projection $q:Y\to X$ is defineable and proper, so the image $Z$ is a closed complex analytic defineable subvariety of $X$ by Remmert's theorem, and therefore it is also algebraic by definable Chow \cite{definechow} (see also \cite{MPT}). Moreover, it contains $\pr_X(U)$, and thus
it contains the smallest algebraic variety containing $\pr_X(U)$, so we may assume $Z=X$.

Consider the family $\mathscr{F}$ of algebraic varieties parametrized by $T_0$.  Let $\Gamma_\mathscr{F}\subset\Gamma$ 
be the subgroup of elements $\gamma$ such that a very general\footnote{Recall that very general means in the  complement of countably 
many proper closed subvarieties.} fiber of $\mathscr{F}$ is stable under $\gamma$. The stabilizer of a very general fiber of $\mathscr{F}$ in $\Gamma$ is then exactly $\Gamma_\mathscr{F}$.  Let $\bTheta$ be the identity component of the 
$\Q$-Zariski closure of $\Gamma_\mathscr{F}$ in $\G$. 

\begin{lemma}

$\bTheta$ is a normal subgroup of $\G$. 

\end{lemma}

\begin{proof}

Let $W'$ be a connected component of $W$ which intersects $X\times\Phi$.  Note that $W'$ is stable under the monodromy group $\Gamma$ of $X$.  Clearly $\mathscr{F}$ is stable under the image $\Gamma_Y$ of $\pi_1(Y)\to\pi_1(X)\to \G(\Z)$ which is finite index in $\Gamma$, and therefore $\Gamma_Y$ is Zariski-dense in $\G$ by Andre-Deligne.

Each element of $\Gamma_Y$ sends a very general fiber of $\mathscr{F}$ to a very general fiber, so by the above remark $\Gamma_\mathscr{F}=\gamma\cdot \Gamma_\mathscr{F}\cdot\gamma^{-1}$ for all $\gamma\in\Gamma_Y$.  It follows that $\bTheta$ is invariant under conjugation by $\Gamma_Y$ and hence by the Zariski closure of $\Gamma_Y$ as well, which is all of $\G$. 

\end{proof}

\begin{prop} $\bTheta$ is the identity subgroup.

\end{prop}

\begin{proof}

Without loss of generality $V$ is a very general fiber of $F$, and hence is invariant by exactly $\bTheta$. Since $\bTheta$ is a $\Q$-group by construction, it follows that $\G$ is isogenous to 
$\bTheta_1\times\bTheta_2$ with $\bTheta_2=\bTheta$ and we have a splitting of weak Mumford--Tate domains $D=D_1\times D_2$ with $D_i=D(\bTheta_i)$.  Replacing $X$ by a finite cover we also have a splitting of the period map \cite[Theorem III.A.1]{GGK}
\[\phi = \phi_1\times\phi_2:X\to \Gamma_1\backslash D_1\times \Gamma_2\backslash D_2.\] Moreover, $\phi_1,\phi_2$ satisfy Griffiths transversality (see the proof of \cite[Theorem III.A.1]{GGK}).
Note that $V\subset X\times D$ by assumption, and as $V$ is invariant under $\bTheta_2$ it is of the form $V_1\times D_2$ where $V_1\subset X\times D_1$.

Consider the period map $X\to \Gamma_1\backslash D_1$, the resulting $W_1\subset X\times D_1$, and the subvariety $V_1\subset X\times D_1$.  Let $U_1$ be the component of $V_1\cap W_1$ onto which $U$ projects.  By assumption the theorem applies in this situation, and as $U_1$ cannot be contained in a proper weak Mumford--Tate subdomain (for then $U$ would as well), we must have
\[\codim_{X\times D_1} (U_1)=\codim_{X\times D_1} (V_1)+\codim_{X\times D_1} (W_1).\] 
Note that the projection $W\to W_1$ has discrete fibers, so $\dim W=\dim W_1$ and $\dim U=\dim U_1$, whereas $\codim V_1=\codim V$, which is a contradiction if $\phi_2$ is non-constant.

\end{proof}

It follows that $V$ is not invariant by any infinite subgroup of $\Gamma$.  The proof of Theorem \ref{main} is then completed by the following lemma, which produces a contradiction:

\begin{lemma}

$V$ is invariant by an infinite subgroup of $\Gamma$.

\end{lemma}

\begin{proof}

Consider the definable set 
$$I:=\{g\in\G(\R) \mid \dim \left(g\cdot V\cap W\cap(X\times\Phi)\right)=\dim U\}.$$ Clearly, $I$ contains $\gamma\in\Gamma$ 
whenever $U$ intersects $X\times \gamma^{-1}\Phi$.  We may assume $1\in I$, and take $x_0\in\Phi$ the second coordinate of a point of intersection of $U$ and $X\times\Phi$.

For any $R>0$, consider the ball $B_0(R)$ centered at $x_0$.  On the one hand, by Theorem \ref{volumebound} we have 
\[\vol\left(U\cap \left(X\times B_{x_0}(R)\right)\right)\gg e^{\beta R}.\]
$U$ is covered with bounded overlaps by $U\cap(X\times\gamma^{-1}\Phi)$ for $\gamma\in \G(\Z)$, so by Proposition \ref{smallvolume} it follows that $I$ has $e^{ \omega (R)}$ integer points.  On the other hand, by Theorem \ref{smallheight} each of these points has height $e^{O(R)}$, and it follows by the Pila-Wilkie theorem that $I$ contains a real algebraic curve $C$ containing arbitrarily
many integer points, in particular at least 2 integer points.  

If $V_c$ is constant in $c$, then $V$ is stable under $C\cdot C^{-1}$. Since $C$ contains at least 2 integer points, it follows that $V$ is stabilized by
a non-identity integer point, completing the proof (since $\Gamma$ is torsion free). So we assume that $V_c$ varies with $c\in C$. Note that since $C$ contains an integer point that 
$\tilde{\phi}(V_c\cap W)$ is not contained in a weak Mumford-Tate subdomain for at least one $c\in C$, and thus for all but a countable subset of $C$ (since there are only countably many families of weak Mumford--Tate subdomains).

We now have 2 cases to consider. First, suppose that $U\subset V_c$ for $c\in C$. 
Then we may replace $V$ by $V_c\cap V_{c'}$ for a generic $c,c'\in C$ and lower $\dim V$, contradicting our
induction hypothesis on $\dim V-\dim U$. 

On the other hand, if it is not true that $U\subset V_c$ for $c\in C$ then $V_c\cap W$ varies with $C$, and so we may set $V'$ to be the Zariski closure of $C\cdot V$. This increases the dimension of $V$ by $1$, but then $\dim V'\cap W = \dim U + 1$ as well, and thus we again contradict our induction hypothesis, this time on $\dim U$. This completes the proof.

 \end{proof}

 \bibliography{biblio.hodge.ax.schanuel}

\begin{thebibliography}{CMSP03}

\bibitem[CDK95]{alghodge}
E.~Cattani, P.~Deligne, and A.~Kaplan.
\newblock On the locus of {H}odge classes.
\newblock {\em J. Amer. Math. Soc.}, 8(2):483--506, 1995.

\bibitem[CKS86]{hodgeasymp}
E.~Cattani, A.~Kaplan, and W.~Schmid.
\newblock Degeneration of {H}odge structures.
\newblock {\em Ann. of Math. (2)}, 123(3):457--535, 1986.

\bibitem[CMSP03]{perioddomain}
J.~Carlson, S.~M\"uller-Stach, and C.~Peters.
\newblock {\em Period mappings and period domains}, volume~85 of {\em Cambridge
  Studies in Advanced Mathematics}.
\newblock Cambridge University Press, Cambridge, 2003.

\bibitem[Gao17]{Gao}
Ziyang Gao.
\newblock Towards the {A}ndre--{O}ort conjecture for mixed {S}himura varieties:
  {T}he {A}x--{L}indemann theorem and lower bounds for {G}alois orbits of
  special points.
\newblock {\em J. Reine Angew. Math.}, 732:85--146, 2017.

\bibitem[GGK12]{GGK}
M.~Green, P.~Griffiths, and M.~Kerr.
\newblock {\em Mumford-{T}ate groups and domains}, volume 183 of {\em Annals of
  Mathematics Studies}.
\newblock Princeton University Press, Princeton, NJ, 2012.
\newblock Their geometry and arithmetic.

\bibitem[GS69]{GS}
P.~Griffiths and W.~Schmid.
\newblock Locally homogeneous complex manifolds.
\newblock {\em Acta Math.}, 123:253--302, 1969.

\bibitem[HT02]{hwangto1}
J.~Hwang and W.~To.
\newblock Volumes of complex analytic subvarieties of {H}ermitian symmetric
  spaces.
\newblock {\em American Journal of Mathematics}, 124(6):1221--1246, 2002.

\bibitem[Kli]{klingler}
B.~Klingler.
\newblock Hodge loci and atypical intersections: conjectures.
\newblock Motives and Complex Multiplication, to appear.

\bibitem[KUY16]{KUY}
B.~Klingler, E.~Ullmo, and A.~Yafaev.
\newblock The hyperbolic {A}x-{L}indemann-{W}eierstrass conjecture.
\newblock {\em Publ. Math. Inst. Hautes \'Etudes Sci.}, 123:333--360, 2016.

\bibitem[MPT17]{MPT}
N.~Mok, J.~Pila, and J.~Tsimerman.
\newblock {A}x-{S}chanuel for {S}himura varieties.
\newblock \href{http://arxiv.org/abs/1711.02189}{\texttt{arXiv:1711.02189}},
  2017.

\bibitem[Pil11]{P}
J.~Pila.
\newblock O-minimality and the {A}ndr\'e-{O}ort conjecture for {$\Bbb C^n$}.
\newblock {\em Ann. of Math. (2)}, 173(3):1779--1840, 2011.

\bibitem[PS03]{definechow}
Y.~Peterzil and S.~Starchenko.
\newblock Expansions of algebraically closed fields. {II}. {F}unctions of
  several variables.
\newblock {\em J. Math. Log.}, 3(1):1--35, 2003.

\bibitem[Sto66]{Stolzenberg}
G.~Stolzenberg.
\newblock {\em Volumes, limits, and extensions of analytic varieties}.
\newblock Lecture Notes in Mathematics, No. 19. Springer-Verlag, Berlin-New
  York, 1966.

\end{thebibliography}
\bibliographystyle{alpha}

\end{document}